\documentclass[a4paper,12pt]{amsart}
\usepackage[active]{srcltx}
\usepackage{graphicx}
\usepackage{amsmath,amssymb,amsthm,amsfonts}
\usepackage{amssymb}
\numberwithin{equation}{section}

\newtheorem{thm}{Theorem}[section]
\newtheorem{cor}[thm]{Corollary}
\newtheorem{lem}[thm]{Lemma}
\newtheorem{prop}[thm]{Proposition}
\newtheorem{rem}[thm]{Remark}

\newcommand\B{{\mathcal B} }
\newcommand\G{{\mathcal G} }
\newcommand\Cz{{C_0(\Omega )}}
\newcommand{\bremark}{\begin{rem} \textup}
\newcommand{\eremark}{\end{rem} }

\newcommand{\cuad}{{\sqcap\kern-.68em\sqcup}}

\newcommand{\R}{{\mathbb{R}}}
\newcommand{\N}{{\mathbb{N}}}

\newcommand{\KK}{{\mathcal {K}}}

\renewcommand{\rho}{\varrho}
\renewcommand{\theta}{\vartheta}

\begin{document}

\title{Sign-changing stationary solutions and blowup for a nonlinear heat equation in dimension two}

\def\shorttitle{Blowup of sign-changing solutions}

\author{Fl\'avio Dickstein}
\address{Instituto de Matem\'atica\\
 Universidade Federal do Rio de Janeiro \\
 Caixa Postal 68530\\
 21944--970 Rio de Janeiro, R.J., Brazil}
\email{fdickstein@ufrj.br}
 \thanks{F. D.  was partially supported by CNPq (Brasil).}

\author{Filomena Pacella}
\address{Dipartimento di Matematica, Universit\`a di Roma ''La Sapienza"\\
P.le A. Moro 2\\
00185 Roma, Italy}
\email{pacella@mat.uniroma1.it}
\thanks{F. P.  was partially supported by PRIN 2009-WRJ3W7 grant(Italy).}

\author{Berardino Sciunzi}
\address{Dipartimento di Matematica, UNICAL,\\
Ponte Pietro  Bucci 31B, \\
87036 Arcavacata di Rende, Cosenza, Italy}
\email{ sciunzi@mat.unical.it}
\thanks{B. S.  was partially supported by ERC-2011-grant: \emph{Epsilon} and PRIN-2011: {\em Var. and Top. Met.}}


\begin{abstract}
Consider the nonlinear heat equation
\begin{equation} \label{fAbs} \tag{\rm NLH}
v_t -\Delta  v= |v|^{p-1} v
\end{equation}
 in the unit ball of $\R^2$, with Dirichlet boundary condition. Let $u_{p,\KK}$ be  a radially symmetric, sign-changing stationary solution having a fixed number $\KK$ of nodal regions. We prove that the solution of~\eqref{fAbs} with initial value $\lambda u_{p,\KK} $ blows up in finite time  if $|\lambda -1|>0$ is sufficiently small and if $p$ is sufficiently large. The proof is based on the analysis of the asymptotic behavior of $u_{p,\KK}$ and of   the linearized operator  $L= -\Delta - p  | u_{p,\KK}  | ^{p-1} $.
\end{abstract}

\maketitle

\section{Introduction} \label{Infro}

Let us consider the nonlinear heat equation
\begin{equation} \label{NLH}
\begin{cases}
v_t- \Delta  v= |v|^{p-1} v, \qquad \text{in} \,\,\,\,\Omega\times (0,T)\\
v=0,\qquad\qquad\qquad \quad\,\,\text{on}\,\,\partial\Omega\times(0,T)\\
v(0)= v_0,\qquad\qquad\quad\,\, \,\,\text{in}\,\,\,\,\Omega\,,
\end{cases}
\end{equation}
where $\Omega\subset \R^N$, $N\in \N$, is a bounded domain, $p>1$, $T\in(0,+\infty]$ and
$$v_0\in \Cz=\{w\in C(\overline{\Omega}), w=0\,\, \text{on}\,\, \partial\Omega\}.$$
It is well known that there exists a unique classical solution of~\eqref{NLH}  which is defined over a maximal time interval $[0, T_{v_0})$. It is also well known that \eqref{NLH} admits both nontrivial global solutions and blowup solutions  for any $p>1$. In fact, given $\varphi \in \Cz$ and $\lambda \in\R$, let us consider $v_\lambda(\varphi) $
 the solution of \eqref{NLH} corresponding to $v_0=\lambda \varphi $. For $|\lambda|$ small, using that the first eigenvalue of the Laplace-Dirichlet operator is positive, it  is easy to construct global sub and supersolutions of \eqref{NLH}, ensuring that $v_\lambda(\varphi) $ is globally defined. On the other hand, $v_\lambda(\varphi) $ has negative energy for large $|\lambda|$ and, as a consequence, it blows up, see~\cite{Ball} or~\cite{Levine}. An interesting question is to understand what happens for intermediate values of $\lambda $. The case of positive functions $\Psi \geq 0$, $\Psi \not\equiv 0$, is better understood. It follows immediately from the maximum principle for the heat equation that there exists $\lambda ^*>0$ such that $v_\lambda(\Psi) $ is global if $0<\lambda<\lambda ^*$ and  $v_\lambda(\Psi) $ blows up if $\lambda>\lambda ^*$. (The borderline case $\lambda =\lambda^*$ may correspond to either globality~\cite{CazenaveH},~\cite{Giga},~\cite{Ni} or to blowup~\cite{Mizoguchi}.)

In other words, defining
\begin{equation*}
   \G=\{v_0\in \Cz, T_{v_0}=\infty \},
\end{equation*}
it holds that $\G^+=\{v_0\in \G, v_0\geq 0\}$ is star-shaped with respect to $0$. (In fact, $\G^+$ is convex.) In general, however, $\G$ is not star-shaped.  In fact, consider the stationary problem
\begin{equation}
	\label{problem}
\begin{cases}
-\Delta\,u\,=|u|^{p-1}u & \text{in $\Omega$,}  \\
u=0 & \text{on $\partial\Omega$.}
\end{cases}
\end{equation}
where $p>1$ and $\Omega$ is the unit ball in $\mathbb{R}^N$, $N>2$. In~\cite{CDW1} the authors showed that there exists $p^*<p_S:=(N+2)/(N-2)$ with the following property. If $u $ is a radial sign-changing solution of the  Lane Emden problem \eqref{problem}
(for subcritical $p$ there are countable many), there exists $\varepsilon>0$ such that if $p^*<p<p_S$ and if $0<|1-\lambda |<\varepsilon $ then $\lambda u \not\in \G$, i.e., $v_\lambda (u)$ blows up in finite time for $\lambda $ slightly greater or slightly smaller then $1$. Note that $u\in\G$, so that $\G$ is not star-shaped. Let us point out that an analogous result was proven for $N=3$ and $p$ close to $1$, see \cite{CDW4}. The results in \cite{CDW1} have been extended to case of general non symmetric domains in \cite{marino}.
Further analysis of the structure of the set $\G$ and of its complementary set
\begin{equation} \label{bset}
   \B=\{v_0\in \Cz, T_{v_0}<\infty \}
\end{equation}
 can be found in~\cite{CDW2} and~\cite{CDW3}.

 The results of \cite{CDW1} and \cite{CDW4} do not apply in the case $N=1$.  In fact, for $N=1$ and $p>1$
$v_\lambda (u)$ is global and converges uniformly to zero if $|\lambda| <1$, while $v_\lambda(u) $ blows up if $|\lambda |>1$. This is due to the anti-periodic structure of the one-dimensional problem, which implies that $v_\lambda(u) $ does not change sign between two consecutive nodes of $u$. In this way, in the one-dimensional case there is no essential difference in considering $u$ with or without a definite sign.

In this paper we treat the case $N=2$, which was left open in~\cite{CDW1}.  We recall that for any $p>1$ and $\KK\in \N$ there exists a unique (up to a sign) radial solution $u_{p,\KK} \in C^2(\overline\Omega)$ of \eqref{problem}  with $\KK$ nodal regions. The main goal of this work is to establish the following result.
\begin{thm} \label{main}
Let $u_{p,\KK}$ be a  sign-changing radial stationary solution of \eqref{NLH} (see \eqref{problem}) with $\KK$ nodal regions. Then there exists $p^*=p^*(\mathcal{K})>1$ and $\varepsilon=\varepsilon(p,\mathcal K)>0$ such that if $p>p^*$ and $0<|1-\lambda |<\varepsilon $, then
\[
\lambda u_{p,\KK}\in \B.
\]
\end{thm}

Our result is analogous in spirit to the one in \cite{CDW1} cited above. In fact, the proofs are based on similar strategies. They are both consequences of the following proposition, which is a particular case of Theorem~2.3 of \cite{CDW3}.
\begin{prop} \label{base}
Let $u$ be a sign changing solution of~\eqref{problem} and let
$\varphi_{1} $ be a positive eigenvector of the self-adjoint operator $L $
given by $L \varphi= - \Delta  \varphi- p|u |^{p-1}\varphi$, for  $\varphi\in H^2(\Omega) \cap H^1_0(\Omega) $. Assume that
\begin{equation}\label{sign}
\int _\Omega u\,  \varphi_{1} \not= 0.
\end{equation}
Then there exists $\varepsilon >0$ such that if $0<|1-\lambda |<\varepsilon $, then the solution $v_\lambda (u)$ of~\eqref{NLH} with the initial value $\lambda u $ blows up in finite time.
\end{prop}

 Proposition \ref{base}  says that the linear instability of the stationary solution expressed by \eqref{sign} yields not only nonlinear instability, but also blowup. A similar result for positive solutions of the nonlinear heat equation and of the nonlinear wave equation may be found in~\cite{ks}. In view of Proposition \ref{base}, Theorem \ref{main} holds if we prove the following:

\begin{thm}\label{mainpreliminarddd}
Given $\KK\ge 2$, let  $u$ be a radial solution to \eqref{problem} having $\KK$ nodal regions.
Then there exists $ p^*= p^*(\KK)$ such that for $p> p^*$
\begin{equation}\nonumber
\int_\Omega\,u\,\varphi_{1}\,>0,
\end{equation}
 where $\varphi_{1}$ is the first positive eigenfunction of the linearized operator $L$ at $u$.
\end{thm}

The proof of Theorem~\ref{main} relies on the fact that, in an appropriate sense, the limit problem of the Lane Emden problem \eqref{problem} is the  Liouville problem
\begin{equation} \label{Liouville}
   \begin{cases}
      -\Delta u=e^u ,  & \text{in } \R^2\\
      e^u\in L^1(\R^2),
   \end{cases}
\end{equation}
see \cite{adi}, \cite{GGP1}, \cite{GGP2}. To be more precise, we consider a suitable scaling $\tilde u $ of $u$, which is defined on a ball
$\tilde\Omega $ of radius $r(p)$ such that $r(p)\to \infty$ as $p\to \infty$. We define as well a  rescaling  $\tilde L$ of the linear operator $L$,  possessing a first eigenvector $\tilde \varphi_1$ associated to a first eigenvalue $\tilde\lambda _1$. Extending $\tilde u $  and  $\tilde \varphi_1$  identically equal to zero outside $\tilde \Omega $, it turns out that
\begin{equation} \label{psi_limit}
   |\tilde u |^{p-1}\tilde u\underset{p\rightarrow \infty}{\longrightarrow } e^{z^*},
\end{equation}
uniformly over the compact sets of $\R^2$, where $z^*$ is the unique radial solution of \eqref{Liouville} such that $z^*(0)=0$ and $\nabla z^*(0)=0$. Moreover, the linearized limit operator $L^*=-\Delta -e^{z^*}$ has a negative first eigenvalue $ \lambda^* _1$ and a positive corresponding eigenfunction $\varphi^*_1$ and
\begin{equation}  \label{lambda_limit}
   \tilde \lambda _1\underset{p\rightarrow \infty}{\longrightarrow }  \lambda^* _1\,,
\end{equation}
\begin{equation} \label{phi_limit}
   \tilde \varphi_1\underset{p\rightarrow \infty}{\longrightarrow } \varphi^*_1\qquad \text{in}\,\,\,L^2(\R^2)\,.
\end{equation}
  Using \eqref{psi_limit} and \eqref{phi_limit} we show that
\begin{equation}\label{int}
   \int_{\tilde\Omega}  |\tilde u |^{p-1}\tilde u\,\varphi^*_1 \underset{p\rightarrow \infty}{\longrightarrow }
   \int_{\R^2} e^{z^*}\varphi^*_1\,.
\end{equation}
Since  both $e^{z^*}$  and $\varphi^*_1$  are positive, the integral at the left hand side of \eqref{int}  is positive for large $p$. By a simple computation, this allows to conclude that \eqref{sign}  holds. Then Theorem~\ref{main} follows from Proposition~\ref{base} and Theorem \ref{mainpreliminarddd}.

To obtain \eqref{psi_limit}-\eqref{int} we exploit the analysis of \cite{GGP2} concerning the case of two nodal regions. For $\KK=2$, the limit problem associated to $u^+$, the positive part of  $u$, is a regular Liouville problem in  the whole space $\R^2$ (while the negative part $u^-$ is associated to a singular  Liouville problem).  Using the results of \cite{GGP2}, we have been able to prove that \eqref{psi_limit} holds for solutions having any fixed number $\KK$ of nodal regions.  There are two crucial steps in the proofs of  \eqref{lambda_limit}-\eqref{int} for general $\KK$, the variational characterization \eqref{variational} of $u$, which is a consequence of the results of \cite{weth}, and the energy estimate \eqref{energy}.

For $N\ge 3$ and subcritical $p<p_S$, it was shown in \cite{CDW4} that $\lambda u\in \B$ if $|1-\lambda|$ and $p_S-p$ are small enough ($\lambda \ne 1$), independently of the number $\KK$ of oscillations of the stationary solution $u$.  We were not able to obtain here an analogous result, since $p$ and $\lambda$ depend on $\KK$ in Theorem \ref{main}.
There is a distinguished difference between the two cases. In the case $N\ge 3$, the limit problem of  \eqref{problem} for $p\to p_S$ is still  the same problem \eqref{problem} for $p=p_S$, which has a (unique, up to dilations and translations) positive regular solution. However, in the present case $N=2$, there is qualitative, other than quantitative, transformation when passing to the limit $p\to \infty$.  This explains why the analysis here is more involved.

The rest of the paper is organized as follows. In Section~\ref{preresults} we obtain some preliminary results that will be useful in the sequel. In particular, we obtain the energy estimate in Proposition~\ref{proenergy} and the variational characterization  in Proposition~\ref{1575}. In Section~\ref{jnvdbndfnb�snbkjcvn}, we carry out an  asymptotic spectral analysis, proving \eqref{lambda_limit} and \eqref{phi_limit}. Finally, in
Section~\ref{sdnfklg7757858} we show \eqref{int}, which yields Theorem~\ref{mainpreliminarddd} and Theorem~\ref{main}.

\section{Preliminary results}\label{preresults}

It is well known that, for $p>1$ and $\KK\ge 1$ \eqref{problem} admits a unique radially symmetric solution $u_{p,\KK}\in C^2(\overline{\Omega})$ having $\KK$ nodal regions and such that $u_{p,\KK}(0) >0$, see e.g. \cite{yana}. In this section we  establish bounds on the energy of  $u_{p,\KK}$ and on its $C_0$ norm which will be crucial for the proof of our main result. These estimates extend those in  \cite{ianni} for the case $\KK=2$.

\begin{prop}\label{proenergy}
There exist $ p^*=p^*(\KK)\in\R$ and $\mathcal{E}=\mathcal{E}(\KK)>0$ such that
\begin{equation}\label{energy}
p\int_{\Omega}\,|u_{p,\KK}|^{p+1}\,dx= p\int_{\Omega}\,|\nabla u_{p,\KK}|^2\,dx\leq \mathcal{E}
\end{equation}
for  $p> p^*$.
\end{prop}

\begin{proof}
Consider the energy functional
\begin{equation}\nonumber
E_p(u)= \frac{1}{2}\|\nabla u\|_2^2-\frac{1}{p+1}\|u\|_{p+1}^{p+1}
\end{equation}
for $u\in H^1_{0,r}(\Omega)$, the space of radial functions of $ H^1_{0}(\Omega)$. Note that, if $u$ is a  solution of~\eqref{problem}, then
\[
E_p(u)=\frac {p-1} {2(p+1)}\int_\Omega |\nabla u|^2.
\]
In this way, Proposition~\ref{proenergy} will be proven once we bound $pE_p(u_{p,\KK})$  uniformly in $p$.  To do so, we first remark that the proofs of Theorem 1.2 and of Theorem 1.4 of \cite{weth} still hold when applied  to the space $ H^1_{0,r}(\Omega)$. As a consequence, we obtain a sequence of distinct solutions of \eqref{problem} $\pm v_{p,j}$, $j\in\mathbb{N}$, such that
\begin{itemize}
\item[a)] $\|v_{p,j}\|_{ H^1_{0,r}(\Omega)}\rightarrow \infty$ as $j\rightarrow \infty$.
\item[b)] $v_{p,1}$ is positive and $v_{p,j}$ changes sign for $j\geq 2$. Moreover $v_{p,j}$ has at most $j$ nodal regions.
    \item[c)] $E_p(v_{p,j})\leq \beta_j$, where
    \begin{equation}\label{beta}
        \beta_j=\underset{\underset{dim (V)\geq j}{V\subset H^1_{0,r}(\Omega)}}{\inf}\sup_{v\in V}\,E_p(v),
    \end{equation}
\end{itemize}
We next observe that,  by the uniqueness (up to a sign) of the radial solution of \eqref{problem} having $j$ nodal regions, we may write that
\begin{equation} \label{app1}
    v_{p,j}= u_{p,j}
\end{equation}
for all $j$. We shall now use $c)$ here above to estimate $pE_p(u_{p,j})$ independently of $p$. Our arguments extend those employed in  \cite{ianni} for the case $j=2$ of two nodal regions.

Given $\KK\in \N$, fix $\alpha_1,\ldots,\alpha_{\KK-1}$ positive numbers satisfying $\alpha_j>\alpha_{j+1}$ for
 $j=1,\ldots,\KK-1$ and set $\alpha _\KK=0$. Consider the $\KK$-dimensional subspace $V_{\KK}^p$ of $H^1_{0,r}(\Omega)$ spanned by the $\KK$ linearly independent functions
$
 g_{p,1},\ldots,g_{p,\KK},
$
 defined in the following way.
 \begin{itemize}
 \item[1)] $g_{p,1}$ is the unique positive radial solution to \eqref{problem} in the ball $$B_p=\{x\in\mathbb{R}^2\,\,:\,\,|x|\leq e^{-\alpha_1p}\}.$$
     \item[2)] For $2\leq j\leq \KK$, $g_{p,j}$ is the unique radial positive solution to \eqref{problem} in the annulus $$A_{p,j}=\{x\in\mathbb{R}^2\,\,:\,\,e^{-\alpha_{j-1}p}\leq |x|\leq e^{-\alpha_{j}p}\}.$$
 \end{itemize}

Let us assume for the moment that there exist $\overline p>1$
 and constants $c_1,\ldots ,c_{\KK}$ such that
 \begin{equation}\label{fil6}
 pE_p(g_{p,j})\leq c_j\qquad \forall p>\overline p,\qquad 1\leq j\leq \KK.
 \end{equation}
Since $g_{p,j}$ belongs to the Nehari manifold
 \begin{equation}\nonumber
 \mathcal{N}_p=\{u\in H^1_{0,r}(\Omega)\setminus \{0\}\,\,:\,\,\|\nabla u\|_2^2=\|u\|_{p+1}^{p+1}\}
 \end{equation}
it is easy to see that $E_p(t g_{p,j})\le  E_p(g_{p,j})$ for all $t\in \R$. By \eqref{fil6},
 \begin{equation*}
 pE_p\big(\overset{\KK}{\underset{j=1}{\sum}} t_jg_{p,j}\big)\leq \overset{\KK}{\underset{j=1}{\sum}}pE_p(g_{p,j})\leq \overset{\KK}{\underset{j=1}{\sum}} c_j,
\end{equation*}
 for all $t_j\in \mathbb{R}$ and $p>\overline p$.  Hence, using \eqref{app1} and $c)$ here above, we get
 \begin{equation}\nonumber
 pE_p(u_{p,\KK})\leq \sup_{v\in V_{\KK}^p}  pE_p(v)\leq \overset{\KK}{\underset{j=1}{\sum}} c_j,
 \end{equation}
 showing \eqref{energy} for any $p>\overline p$. To conclude the proof, it remains to show \eqref{fil6}.

 \bigskip
  We start by estimating $pE_p(g_{p,1})$. Note that
 \begin{equation}\nonumber
    g_{p,1}(|x|)=e^{\frac{2\alpha_1p}{p-1}}w_p(e^{\alpha_1p}|x|),
 \end{equation}
 where $w_p$ is the unique positive solution to \eqref{problem} in the unit ball. Thus,
 \begin{equation}
    \int_{B_p}|\nabla g_{p,1}|^2=e^{\frac {4\alpha_1p} {p-1}} \int_{B_1}|\nabla w_{p}|^2.
 \end{equation}
 Moreover, it follows from Lemma 2.1 of \cite{adi} that
 \begin{equation}\nonumber
 p\int_{B_1}|\nabla w_p|^2\underset{p\rightarrow \infty}{\longrightarrow }8\pi e.
 \end{equation}
 Therefore
 \begin{equation}\nonumber
 pE_p(g_{p,1})=\frac {2p(p+1)} {p-1}\int_{B_p}|\nabla g_{p,1}|^2\underset{p\rightarrow \infty}{\longrightarrow }16\pi e^{4\alpha_1+1}\,,
 \end{equation}
 and this gives \eqref{fil6} for $j=1$.

 \bigskip
  We now estimate $pE_p(g_{p,j})$ for $j\ge 2$. Let $z_{p,j}$ be the positive (radial) solution of
 \begin{equation}\nonumber
 \underset{H^1_{0,r}(A_{p,j})}{\max} \,\{\int_{A_{p,j}}|u|^{p+1}, \,\int_{A_{p,j}} |\nabla u|^2=p^{-1}\,\} =:I_{p,j}.
 \end{equation}

Then $z_{p,j}$ satisfies $-\Delta z_{p,j}=(pI_{p,j})^{-1}\, z_{p,j}^p$, so that $ g_{p,j}=(pI_{p,j})^{-\frac{1}{p-1}}z_{p,j}$. Hence,
 \begin{equation}\label{app2}
p\int_{A_{p,j}}|\nabla g_{p,j}|^2 =(pI_{p,j})^{-\frac{2}{p-1}}.
 \end{equation}
 Next, inspired by the results in \cite{grossi} on the asymptotic  behavior of the radial positive solution in an annulus as $p\rightarrow \infty$, we  set $\Delta _j=\alpha_{j-1}-\alpha_j$ and consider
 \begin{equation*}
     w_{p,j}(x)=(2\pi \Delta _j)^{-\frac {1} {2}}p^{-1}
      \begin{cases}
           \alpha_{j-1}p+\log r,\qquad &e^{-\alpha_{j-1}p}\leq r\leq e^{-\frac{(\alpha_j+\alpha_{j-1})}{2}p},\\
          -\alpha_j\,p-\log r,\qquad &e^{-\frac{(\alpha_j+\alpha_{j-1})}{2}p}\leq r\leq e^{-\alpha_jp},
      \end{cases}
\end{equation*}
where $r=|x|$. Since $w_{p,j}\in H_{0.r}^1(A_{p,j})$ and
$\|\nabla w _{p,j}\|_{L^2(A_{p,j})}^2=p^{-1}$ we get
\begin{equation} \label{app3}
   \int_{A_{p,j}}w _{p,j}^{p+1}\le I_{p,j}.
\end{equation}
Then,
\begin{equation*}
      \int_{A_{p,j}}  w _{p,j}^{p+1}\ge (2\pi)^{-\frac {p-1} {2}} \Delta _j^{-\frac {p+1} {2}}p^{-(p+1)}
      \underset{e^{-\alpha_{j-1}p}}{\overset{e^{-\frac{(\alpha_j+\alpha_{j-1})}{2}p}}{\int}}
      (\alpha_{j-1}p+\log r)^{p+1}r\,dr.
\end{equation*}
Through the change of variables $s=e^{\frac{\alpha_{j-1}+\alpha_j}{2}\,p}\,\,r$, we get
\begin{multline}\label{app4}
   \int_{A_{p,j}}  w_{p,j}^{p+1}
   \geq (2\pi)^{-\frac {p-1} {2}} \Delta _j^{-\frac {p+1} {2}}e^{(\alpha_{j-1}+\alpha_j)\,p}
  \underset{e^{-\frac{p\Delta _j }{2}}}{\overset{1}{\int}}\,
  \left (\frac {\Delta _j} {2}+p^{-1}\log s\right )^{p+1}s\,ds\\
  = 2^{-\frac {3p+1} {2}} \pi^{-\frac {p-1} {2}} \Delta _j^{\frac {p+1} {2}}e^{(\alpha_{j-1}+\alpha_j)\,p}
  \underset{e^{-\frac{p\Delta _j }{2}}}{\overset{1}{\int}}\, \left (1+\frac {2} {p\Delta _j}\log s\right )^{p+1}s\,ds.
\end{multline}
Using the Dominated Convergence Theorem, we obtain
\begin{equation} \label{app5}
   \underset{e^{-\frac{p\Delta _j }{2}}}{\overset{1}{\int}}\,  \left (1+\frac {2} {p\Delta _j}\log s\right )^{p+1}s\,ds
   \underset{p\rightarrow \infty}{\longrightarrow }
   \int_0^1 s^{2(\Delta _j)^{-1}+1}\, ds=\frac {\Delta _j} {2+2\Delta _j}.
\end{equation}
It then follows from \eqref{app2}-\eqref{app5} that
\begin{equation*}
   pE_p(g_{p,j})=\frac {2p(p+1)} {p-1}\int_{A_{p,j}}|\nabla g_{p,j}|^2\le 5\pi (\Delta _j)^{-1}e^{-2(\alpha_{j-1}+\alpha_j)}
\end{equation*}
if $p$ is large enough. This concludes the proof.
\end{proof}
\begin{rem}
 Note that $\min_{j\le \KK} \Delta _j\to 0$ as $\KK\to \infty$. Thus, the energy estimate \eqref{energy} is not independent of $\KK$.
\end{rem}
As a consequence of Proposition \ref{proenergy} and of Theorem 1.2 of \cite{weth} we can show a  nice   variational characterization of the
radial solutions $u_{p,\KK}$  of \eqref{problem}.
\begin{prop}\label{prorem}
We have
\begin{equation}\label{variational}
E_p(u_{p,\KK})=\underset{ \underset{dim (V)\geq \KK}{V\subset H^1_{0,r}(\Omega)}}{\inf}\,\sup_{v\in V}\,E_p(v)\,.
\end{equation}
\end{prop}
\begin{proof}
Denoting by $\chi_1$, $\chi_2$, \dots, $\chi_{\KK}$ the $\KK$ the characteristic functions associated to the $\KK$ disjoint nodal regions of $u_{p,\KK}$, set $u_{p,\KK}^j=u_{p,\KK}\,\chi_j$ and define $V_\KK$ as the subspace generated by $\{u_{p,\KK}^j\}_{j\le \KK}$. Since $E(tu_{p,\KK}^j)\le E(u_{p,\KK}^j)$ for all $t\in \R$, we have that
\begin{equation*}
    E_p\big(\overset{\KK}{\underset{j=1}{\sum}} t_ju_{p,\KK}^j\big)
    \leq \overset{\KK}{\underset{j=1}{\sum}}E_p(u_{p,\KK})=
    E_p(u_{p,\KK}).
\end{equation*}
 From \eqref{beta} we get that $\beta _j\le E_p(u_j)$. The reverse inequality was obtained in the proof of Proposition~ \ref{proenergy} and so \eqref{variational} holds.
\end{proof}

Since for general domains $\Omega$ there could be more solutions having the same number of nodal regions but different energy, as it is the case when $\Omega $ is a ball (see \cite{afta}), a characterization of type \eqref{variational} does not hold for general stationary solutions in $H^1_0(\Omega)$.

 Let  now $\varepsilon_{p,\KK}$ be such that
\begin{equation} \label{epsilon}
   \varepsilon_{p,\KK}^{-2}=pu_{p,\KK}(0)^{p-1}
\end{equation}
and  set
\begin{equation}\label{rp}
   0<r _{p,\KK,1}< r  _{p,\KK,2}<\dots<r_{p,\KK,\KK-1}<1
\end{equation}
the nodal radii of $u_{p,\KK}(|x|)=u_{p,\KK}(r)$, $r=|x|$, in the ball.

\begin{prop}\label{1575}
 We have the following.
\begin{itemize}
\item[i)] $\|u_{p,\KK}\|_{L^\infty(\Omega)}=u_{p,\KK}(0)$.
\item[ii)]  There exist $\underline{c}>0$ and $C(\KK)>0$ such that $\underline{c}\leq u_{p,\KK}(0)\leq C(\KK)$ for all $p>1$.
\item [iii)] $\frac{r _{p,\KK,1}}{\varepsilon_{p,\KK}}\underset{p\rightarrow\infty}{\longrightarrow}  \infty$.
\item [iv)] $\frac{\|u_{p,\KK}\|_{L^\infty(\{|x|\geq r_{p,\KK,1}\})}}{u_{p,\KK}(0)}\underset{p\rightarrow\infty}{\longrightarrow} \theta<\frac {1} {2}$.
\end{itemize}
\end{prop}
\begin{proof}
\noindent Considering $u_{p,\KK}$ as a function of $r=|x|$, it satisfies
\[
   u''_{p,\KK}+\frac {N-1} {r}u'_{p,\KK}+|u_{p,\KK}|^{p-1} u_{p,\KK}=0.
\]
Multiplying the equation by $u'_{p,\KK}$, we get that $F'(r)\le 0$, where
\begin{equation}
   F(r)=\frac {1} {2} |u'_{p,\KK}|^2+\frac {1} {p+1}|u_{p,\KK}|^{p+1}.
\end{equation}
Thus $F$ is nonincreasing. In particular, $F(0)\geq F(r)$ for all $r\geq 0$, which implies that $\|u_{p,\KK}\|_{L^\infty(\Omega)}=u_{p,\KK}(0)$.  This also implies that  the absolute values $M_j$, $j=1,2,\dots,\KK$, of the local maxima of each nodal region of $u_{p,\KK}$ decrease with $j$.

 We next prove the lower bound in $ii)$.  Let us recall that this was shown to be true in Lemma 2.3 of \cite{GGP2}  for the case $\KK=2$  of two nodal regions. This yields the result for general $\KK$, since $u_{p,\KK}(0)>u_{p,2}(0)$. Indeed, for $j<\KK$
\begin{equation} \label{relation}
   u_{p,j}(r)=r_{p,\KK,j}^{\frac {2} {p-1}}u_{p,\KK}(r_{p,\KK,j}r).
\end{equation}
Taking $j=2$, we get $u_{p,\KK}(0)=r_{p,\KK,2}^{-\frac {2} {p-1}}u_{p,2}(0)>u_{p,2}(0)$.

To obtain the upper bound, we see from \eqref{relation} for $j=1$ and from Proposition~\ref{proenergy} that
\begin{multline} \label{upper1}
   p\int_0^1u_{p,1}^{p+1}(r)r\, dr= pr_{p,\KK,1}^{\frac {2(p+1)} {p-1}}\int_0^1u_{p,\KK}^{p+1}(r_{p,\KK,1}r)r\, dr=\\
   pr_{p,\KK,1}^{\frac {4} {p-1}}\int_0^{r_{p,\KK,1}}u_{p,\KK}^{p+1}(s)s\, ds< pr_{p,\KK,1}^{\frac {4} {p-1}}\int_0^1
   u_{p,\KK}^{p+1}(s)s\, ds \le Cr_{p,\KK,1}^{\frac {4} {p-1}}.
\end{multline}
for some $C=C(\KK)$. We next recall that
\begin{equation} \label{upper2}
   \lim_{p\to \infty} p\int_0^1 u_{p,1}^{p+1}(r)r\, dr=\frac{1}{2\pi}\lim_{p\to \infty} p\int_\Omega u_{p,1}^{p+1}\, dx=4 e,
\end{equation}
see \cite{adi}. Using \eqref{upper1} and \eqref{upper2} we conclude that $ r_{p,\KK,1}^{\frac {2} {p-1}}$ is uniformly bounded from below. Finally, we note from \eqref{relation} that
 \begin{equation} \label{quoc}
  r_{p,\KK,1}^{\frac {2} {p-1}}=\frac {u_{p,1}(0)} {u_{p,\KK}(0)}.
 \end{equation}
Since $u_{p,1}(0)\to \sqrt{ e }$, see  \cite{adi}, we conclude that $u_{p,\KK}(0)$ is uniformly bounded from above.
This completes the proof of ii).

To show $iii)$, we use once again \eqref{relation} to write that
\begin{equation}\label{relationr}
   r_{p,\KK,1}=r_{p,\KK,2}r_{p,2,1}
\end{equation}
and that
\begin{equation} \label{relationu}
   u_{p,\KK}^{\frac {p-1}{2}}=u_{p,2}^{\frac {p-1}{2}}r_{p,\KK,2}^{-1}.
\end{equation}
From \eqref{relationr} and \eqref{relationu} we get
\begin{equation}
   \frac {r_{p,\KK,1}}{\varepsilon_{p,\KK}}=\sqrt{p}\, r_{p,\KK,1}u_{p,\KK}^{\frac {p-1}{2}}(0)=\sqrt{p}\, r_{p,2,1}u_{p,2}^{\frac {p-1}{2}}(0)= \frac {r_{p,2,1}}{\varepsilon_{p,2}}.
\end{equation}
Thus the result for general $\KK$ follows from the one for $\KK=2$, which was proven in Proposition 2.7 of  \cite{GGP2}.

It remains to show $iv)$. Since the absolute values of the local maxima of each nodal region of $u_{p,\KK}$   decrease, it follows easily from \eqref{relation} that the quotient in $iv)$ does not depend on $\KK$. For
$\KK=2$, $iv)$ was proven in Theorem 2 of \cite{GGP2}. This closes the proof.
\end{proof}
 The next proposition gives a meaning to the statement that the Lane Emden problem has the Liouville problem as a limit.
\begin{prop}\label{convergence}
Define the rescaled function
\begin{equation*}
z_{p,\KK}=\frac{p}{u_{p,\KK}(0)}(u_{p,\KK}(\varepsilon_{p,\KK}\,x)-u_{p,\KK}(0)),
\end{equation*}
 over the rescaled domain $\Omega_{\varepsilon_{p,\KK}}=\varepsilon_{p,\KK}^{-1}\Omega$ and set $z_{p,\KK}=0$ outside $\Omega_{\varepsilon_{p,\KK}}$. Then
\begin{equation} \label{convergence2}
   z_{p,\KK}\underset{C^1_{loc}(\mathbb{R}^2)}{\longrightarrow} z^*\,,
\end{equation}
where
\begin{equation}\label{213}
z^*\,=\,\log\Big((1+\frac{1}{8}|x|^2)^{-2}\Big)\,.
\end{equation}
is the unique regular solution to the Liouville problem
\begin{equation}
	\label{problemlimite}
\begin{cases}
-\Delta\,z\,=e^z \quad\quad\quad \text{in}  \,\,\,\mathbb{R}^2\\
\int_{\mathbb{R}^2}e^z\,<+\infty,\quad \,\,\,z(0)=|\nabla  z(0)|=0.
\end{cases}
\end{equation}
\end{prop}
\begin{proof}
The proof is similar to that of Theorem 2 in \cite{GGP1}. We outline the main steps for the reader's convenience. Using \eqref{epsilon} it is easy to see that $z_{p,\KK}$ solves
\begin{equation}\nonumber
\texttt{}-\Delta z_{p,\KK} \,=\,\Big|1+\frac{z_{p,\KK}}{p}\Big|^{p-1}\Big(1+\frac{z_{p,\KK}}{p}\Big)\qquad \text{in}\,\,\,\Omega_{\varepsilon_{p,\KK}}\,,
\end{equation}
with $|1+\frac{z_{p,\KK}}{p}|\leq 1$.
By standard regularity theory it follows that $z_{p,\KK}$ is uniformly bounded in $C^2_{loc}(\mathbb{R}^2)$ and hence \eqref{convergence2} holds with $z^*$ satisfying \eqref{problemlimite}. Note that the uniform estimate of the energy obtained in Proposition \ref{proenergy} yields that $\int_{\mathbb{R}^2}e^z\,<+\infty$ (see the proof of Theorem 2 in \cite{GGP1} for details) and \eqref{213} follows by the classification of the solutions to \eqref{problemlimite}.
 \end{proof}
 {\bf \begin{rem}
Here is another argument for the proof of Proposition~\ref{convergence}. It follows from  \eqref{relation}, \eqref{quoc} and \eqref{epsilon} that $z_{p,\KK}=z_{p,1}$ in $\Omega_{\varepsilon_{p,1}}$. This yields \eqref{convergence2} for general $\KK$, since the case of positive solutions $\KK=1$  was shown to be true in \cite{adi}.
\end{rem}}
\section{Asymptotic spectral analysis}\label{jnvdbndfnb�snbkjcvn}
 As discussed in Section~\ref{preresults}, an appropriate rescaling of $u_{p,\KK}$ converges to the solution of the Liouville problem \eqref{problemlimite}. In this section we consider the corresponding linearizations of the Lane Emden and of the Liouville problems and study  their connections.

We first discuss the linearization of the limit problem. For $v\in H^2(\R^2)$ define
\begin{equation}\nonumber
L^*(v)= -\Delta v-e^{z^*} v.
\end{equation}
Consider  the Rayleigh functional
\[
\mathcal{R}(w)= \int_{\mathbb{R}^2} \big(|\nabla w|^2-e^{z^*}w^2\big)\,dx
\]
for $w\in H^1(\mathbb{R}^2)$ and define
\begin{equation}\label{minimo}
\lambda_1^*= \underset{\|w\|_{L^2(\mathbb{R}^2)}=1}{\inf}\mathcal{R}(w)\,.
\end{equation}
We remark that $\lambda_1^*>-\infty$,  since $e^{z^*}$ is bounded.

\begin{prop}\label{proppatritta}
We have the following.
\begin{itemize}
\item[i)]  $\lambda_1^*<0$.
\item[ii)] Every minimizing sequence of \eqref{minimo} has a subsequence which strongly converges in $L^2(\mathbb{R}^2)$ to a minimizer.
\item[iii)] There exists a unique positive minimizer $\varphi_1^*$  to \eqref{minimo} which is radial and radially nonincreasing.  Moreover,  $\lambda_1^*$ is an eigenvalue of $L$ and $\varphi_1^*$ is  an eigenvector associated to $\lambda_1^*$.
\end{itemize}

\end{prop}
\begin{proof}
\noindent A direct computation gives that $e^{z^*}\in H^1(\mathbb{R}^2)$ and that
\[
\mathcal{R}(e^{z^*})=-\frac{1}{2}\int_{\mathbb{R}^2}e^{3z^*}=-\frac{4\pi}{5}\,,
\]
so that $\lambda_1^*$ is negative.  This gives i).

To prove $ii)$ let $w_n$ be a minimizing sequence of \eqref{minimo}. Clearly,  $w_n$ is  bounded in $H^1(\mathbb{R}^2)$.
Therefore, up to a subsequence, it converges weakly to some $w\in H^1(\mathbb{R}^2)$, and strongly in $L^2(\{|x|\leq R\})$
for every $R>0$.
The weak lower semicontinuity of the norm  gives
\[
\int_{\mathbb{R}^2}|\nabla w|^2\leq \underset{n\rightarrow \infty}{\liminf} \int_{\mathbb{R}^2}|\nabla w_n|^2\qquad
\text{and}\qquad
\|w\|_{L^2(\mathbb{R}^2)}\leq 1\,.
\]
Moreover, exploiting the decay properties of $e^{z^*}$, we get
\begin{equation}\nonumber
\begin{split}
&\Big|\int_{\mathbb{R}^2}e^{z^*} (w_n^2-w^2)\Big|\leq \int_{\mathbb{R}^2}e^{z^*} |w_n^2-w^2|=\\
&\int_{\{|x|\leq R\}}e^{z^*} |w_n^2-w^2|\,+\,\int_{\{|x\geq R|\}}e^{z^*} |w_n^2-w^2|\\
 &\leq C\,\|w_n-w\|_{L^2(|x|\leq R)}\,+\, \frac{C}{R^4}\,,
\end{split}
\end{equation}
yielding
\[
\int_{\mathbb{R}^2}e^{z^*} w_n^2\rightarrow \int_{\mathbb{R}^2}e^{z^*} w^2\,.
\]
Therefore $\mathcal{R}(w)\leq \lambda_1^*$, so that $w\ne 0$. Letting
\[
\hat w=\frac{w}{\|w\|_{L^2(\mathbb{R}^2)}}\,,
\]
we  have
\[
\lambda_1^*\le \mathcal{R}(\hat w)=\frac{\mathcal{R}( w)}{\|w\|_{L^2(\mathbb{R}^2)}^2}\leq \frac{\lambda_1^*}{\|w\|_{L^2(\mathbb{R}^2)}^2}\le  \lambda_1^*.
\]
Hence $\|w\|_{L^2(\mathbb{R}^2)}=1$ and $w$ is a minimizer.
This also allows us to deduce that $w_n$ converges to $w$ in $L^2(\mathbb{R}^2)$ so that $ii)$ holds.

 The proof of $iii)$ now uses standard arguments, including a rearrangement procedure  (see \cite{lieb}).
\end{proof}

We next consider the linearization of the Lane Emden problem. In the rest of this paper we fix $\KK\ge 2$ and  we denote for simplicity $u_{p,\KK}$,  $\varepsilon_{p,\KK}$, etc. by $u_p$, $\varepsilon_p$, etc.  Define for $v\in H^2(\Omega)$
\begin{equation}\nonumber
L_{p}(v)=-\Delta v-p|u_{p}|^{p-1}v.
\end{equation}
We denote by $\lambda_1(p)$ the first eigenvalue of $L_p$ in $\Omega$ and by
 $\varphi_{1,p}$ the corresponding positive eigenfunction normalized such that $\varphi_{1,p}>0$ and $\|\varphi_{1,p}\|_{L^2(\Omega)}=1$. In particular, we have
\begin{equation}\label{eqlin}
-\Delta \varphi_{1,p}-p|u_p|^{p-1} \varphi_{1,p}\,=\,\lambda_1(p)\varphi_{1,p}.
\end{equation}
Moreover, $\lambda_1(p)<0$ for any $p>1$, as it is easy to verify.  Let us define $\tilde \varphi_{1,p}$ by
\begin{equation}\nonumber
\tilde \varphi_{1,p}=\varepsilon_p\,\varphi_{1,p}(\varepsilon_p\,x)\qquad \text{in}\quad \Omega_{\varepsilon_p}\,,
\end{equation}
$\tilde \varphi_{1,p}=0$ outside
$\Omega_{\varepsilon_p}$, $\varepsilon_p$  being given by \eqref{epsilon}.  Then $\tilde \varphi_{1,p}$ satisfies
\begin{equation}\nonumber
-\Delta \tilde \varphi_{1,p}\,=\tilde V_p\,\tilde \varphi_{1,p}\,+\,
\tilde\lambda_1(p)\tilde \varphi_{1,p}\,\qquad \text{in}\quad \Omega_{\varepsilon_p}\,,
\end{equation}
where
\begin{equation} \label{Vp}
   \tilde V_p(x)=\frac{|u_p(\varepsilon_p\,x)|^{p-1}}{u_p(0)^{p-1}}=\left |1+\frac{z_p}{p}\right |^{p-1}
\end{equation}
and
\begin{equation} \label{eigenv aluep}
   \tilde\lambda_1(p)=\varepsilon_p^2\lambda_1(p).
\end{equation}
In other words, $\tilde \varphi_{1,p}$ is a first eigenfunction of the operator
\begin{equation}\label{kjdbgkjdgk}
\tilde L_p=-\Delta -\tilde V_p\,I
\end{equation}
in $L^2(\Omega_{\varepsilon_p})$ with $D(\tilde L_p)=H^2(\Omega_{\varepsilon_p} )\cap H^1_0 (\Omega_{\varepsilon_p} )$,
$\tilde\lambda_1(p)$ being the corresponding first eigenvalue.

Extending $\tilde \varphi_{1,p}\equiv 0$ outside $\Omega_{\varepsilon_p}$, we have the following.
\begin{lem}\label{dvbkjdfvbik}
The  set $\{\tilde \varphi_{1,p},p>1\}$ is  bounded in $H_r^1(\mathbb{R}^2)$.
\end{lem}
\begin{proof}
We have that $\|\tilde \varphi_{1,p}\|_{L^2(\mathbb{R}^2)}=1$. In addition, since $\tilde\lambda_1(p)$ is negative and $\|u_p\|_{L^\infty(\Omega)}=u_p(0)$,
\begin{equation}\nonumber
\begin{split}
&\int_{\mathbb{R}^2}|\nabla \tilde \varphi_{1,p}|^2\,= \varepsilon_p^4\int_{\Omega_{\varepsilon_p}}|\nabla \varphi_{1,p}|^2(\varepsilon_p\,x)=\varepsilon_p^2\int_\Omega |\nabla \varphi_{1,p}|^2=\\
&= \varepsilon_p^2\, p\int_\Omega |u_p|^{p-1}\varphi_{1,p}^2+\varepsilon_p^2\,\lambda_1(p)\int_\Omega\varphi_{1,p}^2\\
&\leq \varepsilon_p^2\,p\int_\Omega|u_p|^{p-1}\varphi_{1,p}^2=\frac{1}{u_p(0)^{p-1}}\int_\Omega|u_p|^{p-1}\varphi_{1,p}^2\leq 1\,.
\end{split}
\end{equation}
\end{proof}
\begin{rem}\label{decay}
Applying Strauss Lemma~\cite{strauss} for radial functions of $H_r^1(\mathbb{R}^2)$, we see from Lemma~\ref{dvbkjdfvbik} that $\tilde \varphi_{1,p}(x)\to 0$ as $|x|\to \infty$ uniformly in $p$ and $r=|x|$.
\end{rem}
 We are now ready to discuss the convergence of the eigenvalues $\tilde\lambda_1(p) $.
\begin{thm}\label{fhhdfhsdjsdff}
We have
\begin{equation}\label{limiteigen}
\tilde\lambda_1(p) \underset{p\rightarrow +\infty}{\longrightarrow}\,\lambda_1^*\,.
\end{equation}
\end{thm}
\begin{proof}
We divide the proof in two steps.

\bigskip
\noindent \emph{{\bf Step 1} : For $\epsilon>0$ we have
\begin{center}
$\lambda_1^*\leq \tilde\lambda_1(p)\,+\,\epsilon\,\,\,\,\,\,$ for $p$ sufficiently large.
\end{center}}
To prove this, we see that $\lambda _1^*\le \mathcal{R}(\tilde\varphi_{1,p})$, since $\|\tilde\varphi_{1,p}\|_{L^2(\mathbb{R}^2)}=1$. Thus,
\begin{equation}\label{lambda1}
\begin{split}
\lambda _1^*&\le \int_{\mathbb{R}^2}|\nabla \tilde \varphi_{1,p}|^2-e^{z^*}\tilde \varphi_{1,p}^2= \int_{\Omega_{\varepsilon_p}}|\nabla \tilde \varphi_{1,p}|^2-\tilde V_p\, \tilde \varphi_{1,p}^2-
\int_{\Omega_{\varepsilon_p}} \big(e^{z^*}-\tilde V_p\big )\tilde \varphi_{1,p}^2\\
&= \tilde\lambda_1(p)-\int_{\Omega_{\varepsilon_p}} \big(e^{z^*}-\tilde V_p\big)\tilde \varphi_{1,p}^2\\
&=\tilde\lambda_1(p)-\int_{|x|<R} \big(e^{z^*}-\tilde V_p\big)\tilde \varphi_{1,p}^2-\int_{R<|x|<\varepsilon_p^{-1}}\big (e^{z^*}-\tilde V_p\big )\tilde \varphi_{1,p}^2
\end{split}
\end{equation}
where $R>0$. Using H\"older's inequality, \eqref{energy}, \eqref{epsilon} and \eqref{213} we get
\begin{equation}\nonumber
\begin{split}
&\int_{R<|x|<\varepsilon_p^{-1}} |e^{z^*}-\tilde V_p|\tilde \varphi_{1,p}^2\leq  \|e^{z^*}\|_{L^\infty(\{|x|\geq R\})}+\\
&C \|\tilde \varphi_{1,p}\|^2_{L^\infty(\{|x|\geq R\})}\, (u_p(0))^{-(p-1)}
\left\{\int_{R<|x|<\varepsilon_p^{-1}} u_p(\varepsilon_p x)^{p+1}\right\}^{\frac{p-1}{p+1}}\varepsilon_p^{-\frac{4}{p+1}}\\
&\leq 64R^{-4}+C \|\tilde \varphi_{1,p}\|^2_{L^\infty(\{|x|\geq R\})}\, (u_p(0))^{-(p-1)}
\Big(\frac{\mathcal{E}}{p}\Big)^{\frac{p-1}{p+1}}\varepsilon_p^{-2}\\
&= 64R^{-4}+C \|\tilde \varphi_{1,p}\|^2_{L^\infty(\{|x|\geq R\})}\,
\mathcal{E}^{\frac{p-1}{p+1}} p^{\frac{2}{p+1}}.
\end{split}
\end{equation}
Using that $ \|\tilde \varphi_{1,p}^2\|_{L^\infty(\{|x|\geq R\})}\to 0$ as $R\to \infty$ uniformly in $p$, see Remark~\ref{decay}, and  $ii)$ of Proposition \ref{1575},  we may fix $R$ large enough so that
\begin{equation} \label{lambda2}
   \int_{R<|x|<\varepsilon_p^{-1}} |e^{z^*}-\tilde V_p|\tilde \varphi_{1,p}^2
   \le \epsilon/2
\end{equation}
for all $p>1$.  By \eqref{convergence2}  we get that
$\tilde V_p= (1+\frac{z_p}{p})^{p-1}$ converges uniformly to $e^{z^*}$ on compact sets. In this way, for $R$ fixed as above and $p$ sufficiently large
\begin{equation} \label{lambda3}
   \int_{|x|\leq R} |e^{z^*}-\tilde V_p|\tilde \varphi_{1,p}^2\leq \epsilon/2.
\end{equation}
Step 1 then follows from \eqref{lambda1}, \eqref{lambda2} and \eqref{lambda3}.

\bigskip
\noindent \emph{{\bf Step 2}: Given $\epsilon>0$, we have that
\begin{center}
$\tilde\lambda_1(p)\le \lambda_1^*+\epsilon\quad $ for $p$ sufficiently large.
\end{center}}

\noindent To prove this, let us consider  for $R>0$ a cut-off  regular function $\psi_R(x)=\psi_R(r)$ such that
\begin{itemize}
\item[-] $0\leq\psi_R\leq 1$ with $\psi_R=1$ for $r\leq R$ and $\psi_R=0$ for $r\geq 2R$,
\item[-] $|\nabla \psi_R|\leq 2/R$
\end{itemize}
and set
\[
w_R=\frac{\psi_R\,\varphi_1^*}{\|\psi_R\,\varphi_1^*\|_{L^2(\mathbb{R}^2)}}\,.
\]
We take $R$ such that the ball of radius $2R$ is contained in $\Omega_{\varepsilon_p}$. Since $\Omega_{\varepsilon_p}$
converges to the whole space as $p$ tends to infinity, we can assume that $R$ is arbitrarily large for $p$ large enough.

\noindent From the variational  characterization of $\tilde\lambda_1(p)$ we deduce that
\begin{equation} \label{step21}
\begin{split}
 \tilde\lambda_1(p)& \le \int_{\R^2} |\nabla w_R|^2-\tilde V_p w_R^2 \\& = \int_{\R^2} |\nabla w_R|^2-e^{z^*} w_R^2+
   \int_{\R^2} (e^{z^*} -\tilde V_p) w_R^2
\end{split}
\end{equation}
for all $p>1$.
It is easy to see that $w_R\to \varphi_1^*$ in $H^1(\mathbb{R}^2)$ as $R\to \infty$. Therefore, given $\epsilon>0$ we can fix $R>0$ such that
\begin{equation}\label{step22}
   \int_{\mathbb{R}^2} |\nabla w_R|^2-e^{z^*} w_R^2\leq \lambda_1^*+\epsilon\,.
\end{equation}
For such a fixed value of $R$, we can argue as in Step 1 to obtain that
\begin{equation} \label{step23}
    \int_{\R^2} (e^{z^*} -\tilde V_p) w_R^2\le \epsilon
\end{equation}
for $p$ large enough. Now \eqref{step21}, \eqref{step22} and \eqref{step23} yield Step 2.

Assertion \eqref{limiteigen} follows from Step 1 and Step 2.

\end{proof}

 We may now prove the convergence of the eigenfunctions $\tilde \varphi_{1,p}$.
\begin{cor}\label{hfbjshdvbinv}
$\tilde \varphi_{1,p}$ strongly converges to $\varphi_1^*$ in $L^2(\mathbb{R}^2)$.
\end{cor}
\begin{proof}
Theorem \ref{fhhdfhsdjsdff} shows that  $\tilde \varphi_{1,p}$ is a minimizing sequence for \eqref{minimo}, and so the result follows by  ii) and iii) of Proposition \ref{proppatritta}.
\end{proof}
\section{Proof of Theorem \ref{main}}\label{sdnfklg7757858}
We start with the
\begin{proof}[Proof of Theorem~\ref{mainpreliminarddd}]
Using $\varphi_{1,p}\in H^1_0(\Omega)$ as a test function in \eqref{problem} gives
\begin{equation}\nonumber
\int_\Omega \nabla u_p\cdot\nabla \varphi_{1,p}=\int_\Omega |u_p|^{p-1}u_p\,\varphi_{1,p}\,,
\end{equation}
while using $u_p$ as a test function in \eqref{eqlin} yields
\begin{equation}\nonumber
\int_\Omega \nabla u_p\cdot\nabla \varphi_{1,p}=\int_\Omega p |u_p|^{p-1}u_p\,\varphi_{1,p}+\lambda_1(p)\int_\Omega u_p\,\varphi_{1,p}\,.
\end{equation}
Subtracting  the first equation from the second we obtain
\[
\frac{p-1}{-\lambda_1(p)}\int_\Omega  |u_p|^{p-1}u_p\,\varphi_{1,p}\,=\,\int_\Omega u_p\,\varphi_{1,p}\,.
\]
We may therefore study the sign of $\int_\Omega  |u_p|^{p-1}u_p\,\varphi_{1,p}$
  which is equivalent to  studying the sign of
\[
\frac{1}{u_p(0)^p\,\varepsilon_p}\int_\Omega  |u_p|^{p-1}u_p\,\varphi_{1,p}\,.
\]
In order to prove the result, we will show that
\begin{equation}\label{fgfggfhdhdh}
\frac{1}{u_p(0)^p\,\varepsilon_p}\int_\Omega  |u_p|^{p-1}u_p\,\varphi_{1,p}\underset{p\rightarrow \infty}{\longrightarrow } \int_{\mathbb{R}^2}e^{z^*}\varphi_1^*>0\, \quad \text{as}\,\,\,p\rightarrow\infty.
\end{equation}
To do so, we take $\epsilon>0$ and choose $R>0$ such that
\begin{equation} \label{bound1}
   \int_{|x|\ge R} e^{z^*}\, \varphi_1^*\le \epsilon\,.
\end{equation}
We then write
\begin{equation}\label{bound2}
\begin{split}
&\frac{1}{u_p(0)^p\,\varepsilon_p}\int_\Omega  |u_p|^{p-1}u_p\,\varphi_{1,p}=\frac{1}{u_p(0)^p}\int_{\Omega_{\varepsilon_p}}  |u_p(\varepsilon_p\,x)|^{p-1}u_p(\varepsilon_p\,x)\,\tilde\varphi_{1,p}(x)\\
&=\frac{1}{u_p(0)^p}\int_{|x|<R}  |u_p(\varepsilon_p\,x)|^{p-1}u_p(\varepsilon_p\,x)\,\tilde\varphi_{1,p}(x)\\
&+\frac{1}{u_p(0)^p}\int_{R<|x|<\varepsilon_p^{-1}}  |u_p(\varepsilon_p\,x)|^{p-1}u_p(\varepsilon_p\,x)\,\tilde\varphi_{1,p}(x).
\end{split}
\end{equation}
Using the decay properties of $\tilde\varphi_{1,p}$, see Remark~\ref{decay}, we may take $R$ eventually larger so that
\begin{equation}\label{bound3}
\begin{split}
&\frac{1}{u_p(0)^p}\int_{R<|x|<\varepsilon_p^{-1}}  |u_p(\varepsilon_p\,x)|^{p}\,\tilde \varphi_{1,p}\\
&\leq
C \|\tilde \varphi_{1,p}\|_{L^\infty(\{|x|\geq R\})}\frac {1} {u_p(0)^p}\Big(\int_{R<|x|<\varepsilon_p^{-1}} |u_p(\varepsilon_p\,x)|^{p+1}\Big)^{\frac{p}{p+1}}\varepsilon_p^{-\frac{2}{p+1}}\\
&\leq
C \|\tilde \varphi_{1,p}\|_{L^\infty(\{|x|\geq R\})}\frac {1} {u_p(0)^p}\Big(\int_{\Omega} |u_p|^{p+1}\Big )^{\frac{p}{p+1}}\varepsilon_p^{-2}\\
&\leq
C\|\tilde \varphi_{1,p}\|_{L^\infty(\{|x|\geq R\})}\frac {1} {u_p(0)}p^{\frac {1} {p+1}}\mathcal{E}^{\frac{p}{p+1}}
\leq \epsilon
\end{split}
\end{equation}
for all $p>1$, where we have used \eqref{energy}, $ii)$ of Proposition \ref{1575}, \eqref{epsilon}, H\"older's inequality and a change of variables for the integration.

Moreover,   \eqref{Vp}, \eqref{convergence2}  and Corollary \ref{hfbjshdvbinv}  yield
\begin{equation}\label{bound4}
\begin{split}
&\Big|\int_{|x|\leq R} \Big(\frac{u_p(\varepsilon_p\,x)}{u_p(0)}\Big)^p\,\tilde \varphi_{1,p}-\int_{|x|\leq  R} e^{z^*}\, \varphi_1^*\Big |\\
&=\Big|\int_{|x|\leq R} \Big(1+\frac{z_p}{p}\Big)^p\,\tilde \varphi_{1,p}-\int_{|x|\leq R} e^{z^*}\, \varphi_1^*\Big |\leq \epsilon
\end{split}
\end{equation}
for $p$ eventually larger. Thus \eqref{fgfggfhdhdh} is a consequence of \eqref{bound1}-\eqref{bound4}.
\end{proof}
We finish by proving our main result.
\begin{proof} [Proof of Theorem~\ref{main}]
Theorem~\ref{main} follows immediately from  Theorem~\ref{mainpreliminarddd} and Proposition~\ref{base}.
\end{proof}

\end{document}